\documentclass[reqno]{amsart}

\usepackage{amsmath}
\usepackage{amsthm}
\usepackage{amsfonts}
\usepackage{amssymb}
\usepackage{url}
\usepackage{enumerate}
\usepackage[pdftex,bookmarks=true]{hyperref}
\usepackage{xcolor}
\newcommand\rank{\operatorname{rank}}

\newcommand\R{{\mathbb{R}}}

\renewcommand\P{{\mathbf{P}}}
\newcommand\E{{\mathbf{E}}}

\newcommand\Var{\mathbf{Var}}

\newcommand\Cor{\mathbf{Cor}}

\def\N{\mathbb N}

\def\Inf{\text{Inf}}
\def\rank{\text{rank}}

\theoremstyle{plain}
 \newtheorem{theorem}{Theorem}[section]

 \newtheorem{proposition}[theorem]{Proposition}
 
 \newtheorem{lemma}[theorem]{Lemma}
 \newtheorem{corollary}[theorem]{Corollary}
 \newtheorem{claim}[theorem]{Claim}

\theoremstyle{definition}
\newtheorem{definition}[theorem]{Definition}

\begin{document}

\title[Anti-concentration for polynomials]{Anti-concentration for polynomials of  independent random variables }

\author{Raghu Meka}
\address{Department of Computer Science,
University of California, Los Angeles}
\email{raghum@cs.ucla.edu}

\author{Oanh Nguyen}
\address{Department of Mathematics,  Yale University, New Haven CT 06520, USA} 
\email{oanh.nguyen@yale.edu}

\author{Van Vu}
\address{Department of Mathematics, Yale University, New Haven CT 06520, USA}
\email{van.vu@yale.edu}
\thanks{V. Vu is supported by   NSF  grant DMS-1307797  and AFORS grant FA9550-12-1-0083.}

\begin{abstract}
We prove anti-concentration results for  polynomials of  independent random variables with arbitrary degree. Our results extend the classical Littlewood-Offord result for linear polynomials, and improve several  earlier estimates.

 We discuss applications in two different areas. In complexity theory,  we prove near optimal lower bounds for computing the Parity, addressing a challenge in complexity theory posed by Razborov and Viola, 
and also address a problem concerning OR functions.  In random graph theory, we  derive a general  anti-concentration result on the number of copies of a fixed graph  in a random graph.
\end{abstract}

\maketitle
\section{Introduction} 

Let $\xi$ be a Rademacher random variable (taking value $\pm 1$ with probability $1/2$)
 and  $A=\{a_1,\dots,a_n\}$ be a multi-set in $\R$ (here $n \rightarrow \infty$). Consider the  random sum 

$$S := a_1 \xi_1 + \dots + a_n \xi_n $$ where $\xi_i$ are iid copies of $\xi$.

 In 1943, Littlewood and Offord, in connection with their studies of 
random polynomials \cite{littlewood1943number}, raised the problem of estimating $\P(S \in I)$  for {\it arbitrary} coefficients $a_i$. 
They proved the following remarkable theorem:

\begin{theorem} \label{theorem:LO}   There is a constant $B$ such that the following holds for all $n$. If all coefficients 
$a_i$ have absolute value at least 1, then for any open  interval $I$ of length 1,

$$\P (S  \in I)  \le B n^{-1/2} \log n . $$ \end{theorem}

Shortly after the Littlewood-Offord result, Erd\H{o}s \cite{erdos1945lemma}  removed the $\log n$ term to obtain the optimal bound using an elegant combinatorial proof. 
Littlewood-Offord type results are commonly referred to as anti-concentration (or small-ball) inequalities.  Anti-concentration results have been developed by 
many researchers through decades, and have recently found important applications in the theories of random matrices and random polynomials;
see, for instance,  \cite{nguyen2013small} for  a survey. 

The goal of this paper is to extend Theorem \ref{theorem:LO} to higher degree polynomials. Consider 

\begin{equation}
P (x_1, \dots, x_n)  := \sum_{ S \subset \{1, \dots, n\}; |S| \le d } a_S \prod_{j \in S} x_j. \label{form}
\end{equation}

The first result in this direction, due to Costello, Tao, and the third author,  \cite{costello2006random}, is

\begin{theorem} \label{theorem:LOpoly0}  There is a  constant $B$ such that the following holds for all $d,n$.
 If there are $m n^{d-1} $  coefficients  $a_S$  with  absolute value at least 1, then for any open  interval $I$ of length 1,

$$\P (P(\xi_1, \dots, \xi_n)  \in I) \le B m^{- \frac{1} {2^ {(d^2+d)/2 } } } . $$ \end{theorem}

The exponent $ \frac{1} {2^ {(d^2+d)/2 } }$ tends very fast to zero with $d$, and it is desirable to improve 
this bound. For the case $d=2$, Costello \cite{costello2013bilinear} obtained the optimal bound 
$n^{-1/2+o(1) }$. In a more recent paper \cite{razborov2013real}, Razborov and Viola proved

\begin{theorem} \label{theorem:LOpolyRV}  There is a  constant $B$ such that the following holds for all $d,n$.
 If there are pairwise disjoint subsets $S_1, \dots, S_r$ each of size $d$ such that $a_{S_i}$ have absolute value at least 1
 for all $i$,  then for any open  interval $I$ of length 1,

$$\P (P(\xi_1, \dots, \xi_n)  \in I) \le B  r^{- \frac{1}{ d 2 ^{d+1} }}  . $$ \end{theorem}

This theorem improves  the bound in Theorem \ref{theorem:LOpoly0} to $ m^{- \frac{1}{ d 2 ^{d+1} }} $ via a simple counting argument.

Researchers in analysis also 
considered anti-concentration of polynomials, for entirely different reasons. 
Carbery and Wright \cite {carbery2001distributional} consider polynomials with $\xi_i$ being iid Gaussian and showed 

\begin{theorem} \label{thm:CW}  There is a constant $B$ such that  
$$ \P ( | P (\xi_, \dots, \xi_n) | \le \epsilon \Var (P (\xi_, \dots, \xi_n))^{1/2} ) \le B \epsilon^{1/d} . $$
\end{theorem}

Their  result has been extended  by Mossel, O'donnell and  Oleszkiewicz  \cite{mossel2010noise} to general variables, at a cost of an extra term on the right hand side, which involves 
the regularity of $P$ (see Section 3). 

\vskip2mm

The goal of this paper is to further improve these anti-concentration bounds, with several applications in complexity theory. 
 Our new results will be  nearly optimal in a wide range of parameters. 
Let $[n] = \{1, 2, \dots, n\}$. Following \cite{razborov2013real}, we first introduce a definition 
 
\begin{definition} 
For a degree $d$ multi-linear polynomial of the form \eqref{form}, the \emph{rank} of $P$, denoted by $\rank(P)$,
 is the largest integer $r$ such that there exist disjoint sets $S_1,\ldots,S_r \subseteq [n]$ of size $d$ with $|a_{S_j}| \ge 1$, for $j \in [r]$. 
\end{definition}

Our first main result  concerns the Rademacher case. Let $\xi_i, i=1, \dots, n$ be iid Rademacher random variables. 

\begin{theorem} \label{theorem:LOpoly1}  There is an absolute constant $B$ such that the following holds for all $d,n$. Let $P$ be a polynomial of the form \eqref{form} whose rank $r\ge 2$. Then for any  interval $I$ of length 1,

$$\P (P(\xi_1, \dots, \xi_n)  \in I) \le \min \left (\frac{B d^{4/3}\sqrt {\log r}}{ r^{ \frac{1} {4d+1} } }, \frac{\exp(B d^{2}(\log \log r)^{2})}{\sqrt r}\right ) . $$ \end{theorem}

For the case when $d$ is fixed, it has been conjectured \cite{nguyen2013small}  that $\P (P(\xi_1, \dots, \xi_n) \in I) = O(r^{-1/2} )$. This conjectural bound is a natural generalization of Erdos-Littlewood-Offord result and is optimal, as shown 
by taking $P= (\xi_1 + \dots + \xi_n)^d$, with $n$ even.  For this $P$,  the rank $r= \Theta (n)$ and $\P  ( |P| \le 1/2 ) = \P ( P= 0) = \Theta (n^{-1/2} )$. 
Our result confirms this  conjecture up to the sub polynomial  term $\exp( B d^2 (\log \log r)^2) $. 

 In applications it is important that we can allow the  degree  $d$ tends to infinity with $n$.  
Our bounds in Theorem \ref{theorem:LOpoly1} are non-trivial for degrees up to $c \log r/\log \log r$, for some positive constant $c$.  Up to the $\log \log$ term, this is  as good as it gets, as one cannot hope to get any non-trivial bound  for polynomials of 
degree $\log_2 r$. For example, the degree $d$ polynomial on $2^d \cdot d$ variables defined by $P(\xi) = \sum_{i=1}^{2^d} \prod_{j=1}^d (\xi_{ij}+1)$, where $\xi_{ij}$ are iid Rademacher random variables, has $r = 2^d$ and $\P(P(\xi) = 0) = \Omega(1)$.

Next,  we generalize our result for non-Rademacher  distributions.  As a first step, we consider the  $p$-biased distribution on the hypercube. 
For $p \in (0,1)$, let $\mu_p$ denote the Bernoulli variable with  $p$-biased distribution: $\P_{x\sim \mu_p}(x=0) = 1-p$, $\P_{x\sim \mu_p}(x=1) = p$ and let $\mu_p^{n}$ be the product distribution on $\{0, 1\}^{n}$.

\begin{theorem}\label{th:littlewoodoffordbiasedmain}
There is an absolute constant $B$ such that the following holds.
Let $P$ be a polynomial of the form \eqref{form} whose rank $r\ge 2$. Let $p$ be such that $\tilde r:= 2^{d}\alpha^{d}r\ge 3$ where $\alpha := \min\{p, 1-p\}$. Then for any interval $I$ of length 1, 
$$\P_{x\sim \mu_p^n}(P(x) \in I) \leq \min\left (\frac{B d^{4/3}(\log \tilde r )^{1/2}}{(\tilde r )^{1/(4d+1)}}, \frac{\exp(B d^2(\log \log (\tilde r )^2)}{\sqrt{ \tilde r }}\right ).$$
\end{theorem}


The distribution $\mu_p^n$ plays an essential role in probabilistic combinatorics. For example, it is the ground distribution for the random graphs $G(N,p)$ (with $n:= { N \choose 2}$). 
We discuss an application in the theory of random graphs in the next section. 

Finally, we present a result that applies to virtually all sets of independent random variables, with a weak requirement that these variables do not concentrate on a short interval. 

\begin{theorem}\label{thm:generalist} There is an absolute constant $B$ such that the following holds.
Let $\xi_1, \dots, \xi_n$ be independent (but not necessarily iid) random variables. Let $P$ be a polynomial of the form \eqref{form} whose rank $r\ge 2$.
 Assume that there are positive numbers $p$ and $\epsilon$ such that 
for each $1 \le i \le n$,  there is a number $y_i$ such that $\min \{ \P(\xi_i \le y_i) , \P(\xi_i > y_i ) \}=p$ and $\P (| \xi_i -y_i| \ge 1)  \ge  \epsilon $. 
Assume furthermore that $\tilde r:= (p\epsilon)^d r \ge 3$.  Then for any interval $I$ of length 1 

$$\P(P(\xi_1, \dots, \xi_n) \in I) \leq \min\left (\frac{B d^{4/3}(\log \tilde r )^{1/2}}{(\tilde r )^{1/(4d+1)}}, \frac{\exp(B d^2(\log \log (\tilde r )^2)}{\sqrt{ \tilde r }}\right ).$$
\end{theorem}

Notice that even in the gaussian case, Theorem \ref{thm:generalist} is incomparable to Theorem \ref{thm:CW}.  If we use Theorem \ref{thm:CW} to bound 
$\P ( P \in I)$ for an interval $I$ of length 1, then we need to set $\epsilon = \Var (P)^{-1/2 }$, and the resulting bound  becomes $\frac{B}{ (\Var P) ^{1/2d} } $. 
For sparse polynomials, it is typical that 
$r $ is much larger than $(\Var P) ^{1/d} $ and in this case  our bound is superior.  To illustrate this point, let us fix a constant $d > c >0$ and consider 

$$P := \sum_{S \subset \{1, \dots, n \}, |S|=d} a_S \prod_{ i \in S} x_i $$ where $a_S$ are  iid random Bernoulli variables with $\P( a_S=1) = n^{-c}$. It is easy to show that the following holds with 
probability $1-o(1)$

\begin{itemize} 

\item For any set $X \subset \{1, \dots, n \}$ of size at least $n/2$, there is a subset $S \subset X, |S|=d$,  such that $a_S=1$.
\item The number nonzero coefficients is at most $n^{d-c} $. 
\end{itemize} 

In other words, these two conditions are typical for a sparse polynomial with roughly $n^{d-c}$ nonzero coefficients. 
On the other hand, if the above two conditions holds, then we have $\Var (P) \le n^{d-c} $ and $r \ge n/2d$ (by a trivial greedy algorithm). Our bound  implies that 

$$\P ( P \in I ) \le C(d)  n^{-1/2 +o(1) } $$ while Cabery-Wright bound only gives 

$$\P ( P \in I) \le C(d)  n^{- 1/2 +c/2d }. $$

The rest of the paper is organized as follows. In Section \ref{app} below, we discuss applications in complexity theory and graph theory, with one long proof delayed to Section \ref{app-proof}. Sections \ref{reg-pol} and \ref{reg-lmm} are devoted to some combinatorial lemmas. In Section \ref{main-proof}, we treat polynomials with Rademacher variables. The generalizations are discussed in Section \ref{gen-proof}.  All asymptotic notations are used under the assumption that $n$ tends to infinity. All the constants are absolute, unless otherwise noted.

\section{ Applications} \label{app}
\subsection{Applications in complexity theory}
We use our anti-concentration results to prove lower bounds for approximating Boolean functions by polynomials in the {\sl Hamming metric}. The notion of approximation we consider is as follows. 
\begin{definition}
Let $\epsilon > 0$ and $\mu$ be a distribution on $\{0, 1\}^n$. For a Boolean function $f:\{0, 1\}^n \to \{0, 1\}$ and a polynomial $P:\R^n \to \R$, we say $P$ $\epsilon$-approximates $f$ with respect to $\mu$ \footnote{We drop $\mu$ in the description when it is clear from context or if it is the uniform distribution.} if 
$$\P_{x\sim\mu}(P(x) = f(x)) > 1-\epsilon.$$
We define $d_{\mu,\epsilon}(f)$ to be the least $d$ such that there is a degree $d$ polynomial which $\epsilon$-approximates $f$ with respect to $\mu$.
\end{definition}

An alternate ({dual}) way to view the above notion is in terms of distributions over low-degree polynomials---``randomized polynomials''---which approximate the function in the worst-case. In particular, by Yao's min-max principle, $d_{\mu,\epsilon}(f) \leq d$ for every distribution $\mu$ if and only if there exists a distribution $\mathcal D$ over degree at most $d$ polynomials which approximates $f$ in the worst-case: for all $x$, $\P_{P \sim \mathcal D}[P(x) = f(x)] > 1- \epsilon$.

Approximating Boolean functions by polynomials in the Hamming metric was first considered in the works of Razborov \cite{razborov1987lower} and Smolensky \cite{smolensky1987algebraic} over fields of finite characteristic as a technique for proving lower bounds for small-depth circuits. This was also studied in a similar context over real numbers by the works of \cite{beigel1991perceptron}, \cite{aspnes1994expressive}; the latter work uses them to prove lower bounds for $AC(0)$. More recently, in a remarkable result, Williams \cite{williams2014faster} (also see \cite{williams2014polynomial, abboud2015more}) used polynomial approximations in Hamming metric for obtaining the best known algorithms for all-pairs shortest path and other related algorithmic questions. Here, we study lower bounds for the existence of such approximations. 

{\bf Approximating Parity.} Let $par_n:\{0, 1\}^n \to \{0, 1\}$ denote the parity function: $par_n(x) = x_1 \oplus x_2 \oplus \cdots \oplus x_n$ (where arithmetic is mod $2$). 

In \cite{razborov2013real}, Razborov and Viola introduced another way to look at this problem. For two functions $f, g : \{0,1\}^n \rightarrow \R$,   define their "correlation" to be the quantity
$$ Cor_n (f, g) = \P _x (f(x) =g(x) ) -1/2 , $$ 
where $x$ is uniformly distributed over  $\{0,1\}^n$.  They highlighted the following challenge 

{\bf Challenge.} Exhibit an explicit boolean function $f: \{0,1 \}^n \rightarrow \{0,1 \}  $ such that for any real polynomial $P$ of degree 
$\log_2 n$, one has 
$$\Cor_n (f,P) \le  o( 1/\sqrt n ). $$ 

This challenge is motivated by studies in complexity theory and has connections to many other problems, such as the famous rigidity problem;
see \cite{razborov2013real} for more discussion. 

The Parity function seems to be a natural candidate
in problems like this. Razborov and Viola, using Theorem \ref{theorem:LOpolyRV}, proved

\begin{theorem} \label{theorem:RV1}  \cite{razborov2013real} For all sufficiently large $n$, $\Cor_n (par_n, P) \le 0$ for any real polynomial $P$ of degree at most $\frac{1}{2} \log_2 \log_2 n $.
\end{theorem} 

With Theorem \ref{theorem:LOpoly1}, we obtain the following improvement, which gets us within the Challenge by a $\log \log n$ factor.

\begin{theorem} \label{theorem:RV2} For all sufficiently large $n$, $\Cor_n (par_n, P) \le 0$ for any real polynomial $P$ of degree at most $\frac{ \log  n}{15 \log \log n }$.
\end{theorem} 

\begin{proof}
Let $d$ be the degree of $P$. Following the arguments in the proof of \cite[Theorem 1.1]{razborov2013real},
we can assume that $P$ contains at least $\sqrt n$ pairwise disjoint subsets $S_i$ each of size $d$ and non-zero coefficients. It suffices to show that  the probability that $P$ outputs a boolean value is at most $1/2$. By replacing $P$ by $q(x_1, \dots, x_n) := P((x_1+1)/2, \dots, (x_n+1)/2)$, one can convert the problem into polynomial of the same degree defined on $\{\pm 1\}^{n}$, in other words, on Rademacher variables.  Then by Theorem \ref{theorem:LOpoly1}, this probability is bounded by $2B\frac{d^{4/3}\log ^{1/2}n}{n^{1/(8d+2)}}$. This is less than $1/2$ for every $d\le \frac{\log n}{15\log \log n}$ when $n$ is sufficiently large.
\end{proof}

{\bf Approximating AND/OR.} One of the main building blocks in obtaining polynomial approximations in the Hamming metric is the following result for approximating the OR function\footnote{$OR(x_1,\ldots,x_n)$ is $1$ if any of the bits $x_i$ is non-zero.}. 

\begin{claim}
For all $\epsilon \in (0,1)$ and distributions $\mu$ over $\{0, 1\}^n$, there exists a polynomial $P:\R^n \to \R$ of degree at most $O((\log n)(\log 1/\epsilon))$ such that
$\P_{x\sim \mu}(P(x) = OR(x)) > 1 - \epsilon$.
\end{claim}

By iteratively applying the above claim, Aspnes, Beigel, Furst, and Rudich \cite{aspnes1994expressive} showed that $AC(0)$ circuits of depth $d$ have $\epsilon$-approximating polynomials of degree at most $O(((\log s)(\log(1/\epsilon)))^d \cdot (\log(s/\epsilon))^{d-1})$.  We prove that the following lower bound for such approximations: 

\begin{theorem}\label{th:orapprox}
There is a constant $c > 0$ and a distribution $\mu$ on $\{0, 1\}^n$ such that for any polynomial $P:\{0, 1\}^n \to \R$ of degree $d < c (\log \log n)/(\log \log \log n)$,
$$\P_{x\sim\mu}(P(x) = OR(x)) < 2/3.$$
\end{theorem}

To the best of our knowledge no $\omega(1)$ lower bound was known for approximating the OR function.
We give an explicit distribution (directly motivated by the upper bound construction in \cite{aspnes1994expressive}) under which OR has no $1/3$-error polynomial approximation. 
The distribution $\mu$ on $\{0, 1\}^n$ we consider is as follows: 
\begin{enumerate}
\item With probability $1/2$ output $x =0$.
\item With probability $1/2$ pick an index $i \in [D]$ uniformly at random and output $x \gets \mu_{2^{-a^i}}^n$ for some suitably chosen parameters $a, D$. 
\end{enumerate}
The analysis then proceeds at a high level as in the lower bound for parity. However, we need some extra care with the inductive argument as unlike for parity, we can't consider arbitrary fixings of subsets of coordinates of the OR function. We get around this hurdle by instead only considering fixing parts of the input to $0$ and decreasing the bias $p$ to make sure that these coordinates are indeed set to $0$ with high probability.  The details are defered to Section 7.

\subsection{The number of small subgraphs  in a random graph} 

Consider the Erd\H{o}s-R\'enyi random graph $G(N,p)$. Let $H$ be a small fixed graph (a triangle or $C_4$, say). The problem of counting the number of copies of $H$ in $G(N,p)$ is 
a fundamental topics in the theory of random graphs (see, for instance, the text books   \cite{Bollobasbook, Jansonbook}).  In fact,  one can talk 
about a more general problem of counting the number of copies of $H$ in a random subgraph of any  deterministic graph $G$ on $N$ vertices, formed by choosing each edges of $G$ with probability $p$.
  We denote the $F(H,G,p)$ this random variable.  In this setting we understand that $H$ has constant size, and the size of $G$ tends to infinity. 

It has been noticed  that $F$ can be written as a polynomial  in term of the edge-indicator random variables.
For example, the number of $C_4$ (circle of length $4$)  is 

$$ \sum_{ i,j,k,l  } \xi_{ij} \xi_{jk} \xi_{ kl} \xi_{li}  $$ where the summation is over all quadruple $ijkl$  which forms a $C_4$ in $G$   and the Bernoulli random variable 
 $\xi_{ij}$ represents the edge $ij$.  Clearly, any  polynomial of this type  has  $n = e(G)$ 
iid Bernoulli $p$-bias variables $\xi_{ij}$, and its degree equals the number of edges of $H$.  The rank $r$ of $F$ is exactly the size of the largest collection of edge disjoint copies of 
$H$ in $G$.  

The polynomial representation has been useful in proving {\it concentration} (i.e.{\it large deviation} ) results for $F$ 
(see \cite{KimVu, Vusurvey}, for instance).  Interestingly, it has turned out that one can also use this to derive anti-concentration result, in particular bounds on the probability that 
the random graph has exactly $m$ copies of $H$.

By  Theorem \ref{th:littlewoodoffordbiasedmain}, we have


\begin{corollary}  \label{Hcopy} 
Assume that $p$ is a constant in $(0,1)$. Then for fixed $H$ and any integer $m$ which may depend on $G$ 

$$\P ( F(H,G ,p) = m )  \le r^{-1/2 +o(1)} , $$  where $r$ is the size of the largest collection of edge-disjoint copies of $H$ in $G$.  In particular, if $G=K_n$, then 

$$\P ( F(H,K_n ,p) = m )  \le n^{-1/2 +o(1)}. $$  
\end{corollary}

A similar argument can be used to deal with the number of {\it induced} copies of $H$, which can be also written as a polynomial with 
degree at most ${v \choose 2}$, with $v$ being the number of vertices of $H$.    Details are left out as an exercise. 

Finally, let us mention that in a recent paper \cite{GK}, Gilmer and Kopparty  obtained a precise estimate for $\P( F(H, K_n, p) =m)$ in the case when $H$ is a triangle. \footnote{We would like to thank J. Kahn for pointing out this reference.} Their approach relies on a  careful treatment of the characteristic function. It remains to be seen if this method  applies to our more general setting.

\section {Regular polynomials}\label{reg-pol}

Our proofs  of anti-concentration bounds use the techniques developed in the context of bounding the \textit{noise sensitivity} of \textit{polynomial threshold functions} in the works \cite{diakonikolas2014regularity, harsha2014bounding, kane2014pseudorandom}. In particular, we use the concept of \textit{regular polynomials}, the invariance principle of Mossel, O'donnell, and Oleszkiewicz \cite{mossel2010noise}, and the \textit{regularity lemma} of \cite{diakonikolas2014regularity, harsha2014bounding}.  In this and the following section, we discuss these tools.



To start,  we  define  regular polynomials and discuss an anti-concentration result for them. The \textit{influence} of the $i$-th variable on $P$ is defined to be $\Inf_i = \Inf_i(P)=\sum_{i\in S} a_S^{2}$. Since $ \Var(P) = \sum_{S\neq \emptyset} a_S^{2}$, we have 
\begin{equation}
\Var(P)\le \sum_{i=1}^{n} \Inf_i \le d\Var(P).\label{total_inf}
\end{equation}
Assume the random variables are ordered such that $ \Inf_1\ge  \Inf_2\ge\dots \ge\Inf_n$. 
Let $ \tau>0 $, the $ \tau$-\textit{critical index} of $ P $ is the least $ i $ such that $ \Inf_{i+1}\le \tau\sum_{j=i+1}^{n}\Inf_j$. If it does not hold for any $i$, we say that the $P$ has $\tau$-critical index $\infty$. If $ P $ has $ \tau $-critical index 0, we say that $ P $ is $ \tau $-regular.
The following is a corollary of strong results from \cite{carbery2001distributional} and \cite{mossel2010noise}. 

\begin{proposition} \label{regular}
	Let $P$ be a non-constant polynomial of the form \ref{form}. Let $\tau> 0$. If $ P $ is $ \tau $-regular, then $ \P(|P(\xi_1, \dots, \xi_n)|\le \alpha)\le \frac{Cd\alpha^{1/d}}{(\Var(P))^{1/2d}}+  Cd\tau^{1/(4d+1)}$ for every $ \alpha>0 $.
\end{proposition}

\begin{proof}

Let $\tilde \xi_1, \dots, \tilde \xi_n$ be independent standard Gaussian variables. Notice that $$\Var(P(\xi_1\, \dots, \xi_n)) = \Var(P(\tilde \xi_1, \dots, \tilde \xi_n)).$$

Our settings satisfy the Hypothesis {\bf H$4$} of \cite[Theorem 3.19]{mossel2010noise} with $r = 4$. Using that theorem, one obtains
\begin{eqnarray}\label{moo}
\P(|P(\xi_1, \dots, \xi_n)|\le \alpha)&\le& \P(|P(\tilde \xi_1, \dots,\tilde \xi_n)|\le \alpha) + Cd\tau^{1/(4d+1)}.
\end{eqnarray}  

Now, for Gaussian case, it was proved in \cite[Theorem 8]{carbery2001distributional} that for every $ \alpha>0 $,
	\begin{equation}\label{gau}
	\P(|P(\tilde{\xi}_1, \dots, \tilde{\xi_n})|\le \alpha)\le C\frac{d\alpha^{1/d}}{(\Var (P))^{1/2d}}.
	\end{equation}

Combining \eqref{moo} and \eqref{gau}, we get the desired bound.
\end{proof}

\section{ A regularization  lemma}\label{reg-lmm}

Proposition \ref{regular} would yield our desired bound in Theorem \ref{theorem:LOpoly1}  if $\tau$ is small (say at most $r^{-1}$). However, there is
no guarantee for this assumption.  In order to go from the regular case to the general case, we will use the following regularization  lemma, whose proof is a slight modification of \cite[Theorem 1.1]{diakonikolas2014regularity} (the version below gives us  better quantitative bounds in our applications). The main idea is to condition on the random variables with large influence. With high probability, the resulting polynomial is either regular or dominated by its constant part.

For a set $S\subset [n]$, we consider a random assignment $\rho\in \{\pm 1\}^{|S|}$ which assigns values $\pm 1$ to variables $(\xi_i)_{i\in S}$. We say that ``$\rho$ fixes $S$". For each such $\rho$, the polynomial $P$ becomes a polynomial of $(\xi_{i})_{i\notin S}$ which is denoted by $P_{\rho}$. We write $P_{\rho} = P^{*}(\rho) + q_{\rho}(\xi_i)_{i\notin S}$ where $P^*$ is the constant part of $P_{\rho}$ consisting of monomials of $(\xi_i)_{i\in S}$ only. For $C>0$ and $0<\beta<1$, we say that $P_{\rho}$ is \textit{$(C,\beta)$-tight} if 
\begin{equation}
\sqrt{\Var_{(\xi_{i})_{i\notin S}} (q_{\rho})} \le |P^*(\rho)|\left(C(\log\frac{1}{\beta})\right)^{-d/2},\label{q1}
\end{equation}
and
	\begin{equation}
	\P_{(\xi_{i})_{i\notin S}}\left (|q_\rho|\le \frac{1}{2}|P^*(\rho)|\right )\ge 1-\beta\label{qp*1}.
	\end{equation} 
Note that it is always true that $\E_{(\xi_{i})_{i\notin S}} q_{\rho} = 0$. We shall see later that \eqref{q1} actually implies \eqref{qp*1}.
	
\begin{proposition}\label{regularity lemma}
	There exist absolute constants $C$ and $C'$ such that the following holds true. Let $P(\xi_1, \dots, \xi_n)$ be a a degree-$d$ polynomial, let $0<\tau, \beta<\frac{1}{3}$. Let $\alpha = C(d\log\log 1/\beta+d\log d)$ and $\tau' = (C'd\log d\log\frac{1}{\tau})^{d}\tau$. Let $M\in \N$ such that $M\frac{\alpha}{\tau}\le n$. Then, there exists a decision tree of depth at most $M\frac{\alpha}{\tau}$ with $P$ at the root, variables $\xi_i$'s at each internal node, and a degree-$d$ polynomial $P_{\rho}$ at each leaf $\rho$, with the following property: with probability at least $1 - (1-\frac{1}{2C^{d}})^{M}$, a random path from the root $P$ reaches a leaf $\rho$ such that $P_{\rho}$ is either $\tau'$-regular or \textit{$(C,\beta)$-tight}.

\end{proposition}
\begin{proof}
	
	First, we consider the case when the $\tau$-critical index of $P$ is large. For a positive integer $K$, denote by $[K]$ the set $\{1, \dots, K\}$.
	\begin{lemma}\label{large critical index}
		There exists a constant $C$ such that the following holds true. Let $0<\tau, \beta<\frac{1}{3}$ be deterministic constants that may depend on $n$. Suppose that $P$ has $\tau$-critical index at least $K = \frac{\alpha}{\tau}$, where $\alpha = C(d\log\log 1/\beta+d\log d)$. Then for at least $\frac{1}{2C^{d}}$ fraction of restrictions $\rho$ fixing $[K]$, the polynomial $P_\rho$ is $(C, \beta)$-tight.
	\end{lemma}
	Roughly speaking, the $(C,\beta)$-tightness asserts that the resulting polynomial $P_{\rho}$ has large constant term, compared to the random part, and therefore, it concentrates around the constant part.
	\begin{proof}
		Since the proof is completely the same as the proof of \cite[Lemma 3.5]{diakonikolas2014regularity}, we only provide a sketch here. 
		Without loss of generality, assume that $\Var(P) = 1$.
		We first show that 
		\begin{equation}
		\P_{\rho}(|P^*(\rho)|\ge \frac{1}{2C^{d}})\ge \frac{1}{C^{d}}\label{lmm4DSTW}
		\end{equation}
		where by $\P_\rho$ we mean the probability with respect to $\xi_1, \dots, \xi_K$.
		Observe that $\Var_{\rho}(P^*(\rho))=\sum_{\emptyset\neq S\subset [K]} a_S^{2}\le \Var(P)=1$. Moreover, by definition of critical index, 
		\begin{equation}
		\sum_{i\notin [K]} \Inf_i(P)\le(1-\tau)^{K}\sum_{i=1}^{n}\Inf_i(P)\le de^{-\alpha}\le \frac{1}{2}.\label{suminf}
		\end{equation}
		 Hence, $1\ge \Var_{\rho} (P^*(\rho)) = \Var(P) - \sum_{S\subset [n], S\nsubseteq [K]} a_S^{2}\ge 1 - \sum_{i\notin [K]}\Inf_i(P)\ge\frac{1}{2}$. Then, we use the following Theorem
		\begin{theorem}(\cite{austrin2011randomly}, \cite{dinur2006fourier}, also \cite[Theorem 2.5]{diakonikolas2014regularity}) There is a universal constant $C_0>1$ such that for any non-zero degree-$d$ polynomial $P: \{-1, 1\}^{n}\to \R$ with $\E (P) = 0$, we have 
			$$\P\left (P>\frac{\sqrt{\Var(P)}}{C_0^{d}}\right )>\frac{1}{C_0^{d}}.$$
		\end{theorem}
		
		Let $C\ge C_0^{2}$. Applying the above Theorem to $P^*(\rho) - \E_{\rho}P^*(\rho)$ if $\E_{\rho}P^*(\rho)\ge 0$ and $-P^*(\rho) +\E_{\rho}P^*(\rho)$ otherwise gives \eqref{lmm4DSTW}.

		Next, we show that 
		\begin{equation}
		\P_{\rho}\left (\Var (q_{\rho})>\frac{1}{(2C^{d})^{2}}\left(C(\log\frac{1}{\beta})\right)^{-d}\right )\le \frac{1}{2C^{d}}.\label{lmm5DSTW}
		\end{equation}
		Indeed, let $Q(\rho) = \Var (q_{\rho})$. By triangle inequality and Bonami-Beckner inequality (see, for instance, \cite[Theorem 2.1]{diakonikolas2014regularity}, or \cite{bonami1970etude}, \cite{gross1975logarithmic}), one can show that $||Q(\rho)||_{2} = \sqrt{\E_{\rho} Q^{2}(\rho)} \le 3^{d}\sum_{i>K}\E_{\rho}\Inf_{i}(P_\rho)= 3^{d}\sum_{i> K}\Inf_i(P)\le 3^{d}de^{-\alpha}$ where the last inequality is just \eqref{suminf}. From this, we use the following Theorem
		\begin{theorem}(\cite{austrin2011randomly}, \cite{dinur2006fourier}, also \cite[Theorem 2.2]{diakonikolas2014regularity})\label{thm6}
			Let $P:\{-1, 1\}^{n}\to \R$ be a degree-$d$ polynomial. For any $t>e^{d}$, we have
			\begin{equation}
			\P(|P|> t||P||_2)\le \exp(-\Omega(t^{2/d})).\nonumber
			\end{equation}	
		\end{theorem}  
Using this Theorem for the polynomial $Q$ and $t = d^{d}C^{d}\log^{d}C$, we get \eqref{lmm5DSTW}.		

From \eqref{lmm4DSTW} and \eqref{lmm5DSTW}, with probability at least $\frac{1}{2C^{d}}$ over all possible $\rho$, \eqref{q1} happens. For each such $\rho$, using Theorem \ref{thm6} for $q$, we obtain
		\begin{eqnarray}
		\P_{\xi_{K+1}, \dots, \xi_n}(|q_{\rho}|\ge \frac{1}{2}|P^*(\rho)|)\le \P_{\xi_{K+1}, \dots, \xi_n}\left (|q_{\rho}|\ge \frac{1}{2}\left (C\log\frac{1}{\beta}\right )^{d/2}||q_{\rho}||_2\right )\le \beta,\nonumber
		\end{eqnarray}
		which gives \eqref{qp*1} and completes the proof of Lemma \ref{large critical index}.
\end{proof}

	Next, we consider the case when $P$ has small critical index. We'll use the following Lemma \cite[Lemma 3.9]{diakonikolas2014regularity} which asserts that by assigning values to the random variables with large influences, with significant probability, one gets a regular polynomial.
	\begin{lemma}\label{small critical index}
Let $C$ be the constant in Lemma \ref{large critical index}. There exists an absolute constant $C'$ such that the following holds. Let $0<\tau<\frac{1}{3}$. Assume that $P$ has $\tau$-critical index $k\in [n]$. Let $\rho$ be a random restriction fixing $[k]$, and $\tau' = (C'd\log d\log\frac{1}{\tau})^{d}\tau$. With probability at least $\frac{1}{2C^{d}}$ over the choice of $\rho$, the restricted polynomial $P_\rho$ is $\tau'$-regular.
	\end{lemma} 
	
	Combining Lemmas \ref{large critical index} and \ref{small critical index}, we get
	\begin{lemma}\label{tree}
		Let $P(\xi_1, \dots, \xi_n)$ be a a degree-$d$ polynomial, $0<\tau, \beta<\frac{1}{3}$. Let $\alpha = C(d\log\log 1/\beta+d\log d)$ and $\tau' = (C'd\log d\log\frac{1}{\tau})^{d}\tau$. Assume that $\Inf_1\ge \Inf_2\dots\ge\Inf_n$. Then one of the following holds true.
		\begin{enumerate}
			\item $P$ is $\tau$-regular.
			\item The $\tau$-critical index of $P$ is at least $\frac{\alpha}{\tau}$ and the conclusion of Lemma \ref{large critical index} holds.
			\item The $\tau$-critical index of $P$ is $k<\frac{\alpha}{\tau}$ and the conclusion of Lemma \ref{small critical index} holds.
		\end{enumerate}
	\end{lemma}

	Now, we are ready for the proof of Proposition  \ref{regularity lemma}. The strategy is to apply Lemma \ref{tree} repeatedly $M$ times. At first, if $P$ is not $\tau$-regular, we apply Lemma \ref{tree} to $P$ and obtain an initial tree of depth at most $\frac{\alpha}{\tau}$. We know that at least $\frac{1}{2C^{d}}$ fractions of the restricted $P_{\rho}$ are "good", i.e., either $\tau'$-regular or $(C, \beta)$-tight. We keep them as leaves of our final tree and leave them untouched during the next stages. At the second stage, for each of the remaining "bad" polynomials $P_{\rho}$, we order the unrestricted variables in decreasing order of their influences in $P_{\rho}$, and then apply lemma \ref{tree} to it. Note that probability of reaching a bad leaf in this second tree is at most $(1-\frac{1}{2C^{d}})^{2}$. Continuing in this manner $M$ times, we get the desired tree and complete the proof of Theorem \ref{regularity lemma}.
\end{proof}

\section {Proof of Theorem \ref{theorem:LOpoly1}}\label{main-proof}

The high-level argument for the first bound of \ref{theorem:LOpoly1} is as follows. If the polynomial is sufficiently \textit{regular}, we apply the \textit{anti-concentration} property of regular polynomials; the latter property in turn follows from the invariance principle and a similar anti-concentration property for polynomials with respect to the Gaussian distribution. 

To complete the argument, we use the regularity lemma  which shows that any polynomial can be written as a small-depth decision tree where most leaves are labeled by polynomials which are either (1) Regular or (2) Polynomials which are fixed in sign with high probability over a uniformly random input. In the first case, you get a regular polynomial of high rank (as the tree is shallow) and we apply the previous argument. In the second case, we argue directly that the probability of taking the value $0$ is small. 

To prove the second bound of \ref{theorem:LOpoly1}, we follow the same conceptual approach but adopt a more careful analysis following the work of Kane \cite{kane2014correct}. We defer the details to the actual proof. 

\subsection{First bound}

	Without loss of generality, we can assume that $I$ is centered at 0 and $r$ is larger than some constant. We can also assume that $d\le \frac{2\log r}{\log \log r}$ because otherwise $dr^{-1/(4d+1)}\ge 1$ and the desired bound becomes trivial.

	Let $\tau \in (0, \frac{1}{3})$ and let $\beta = \frac{1}{r}$. We will use  Proposition \ref{regularity lemma} to reduce to the regular case. Let $\alpha$, $\tau'$ be as in that Proposition, i.e., $\alpha = C(d\log\log \frac{1}{\beta}+d\log d)$ and $\tau' = (C'd\log d\log\frac{1}{\tau})^{d}\tau$. Let $M = \lfloor\frac{r\tau}{2\alpha}\rfloor$. Call a leaf of the decision tree \emph{good} if $P_\rho$ is either $\tau$-regular or $(C,\beta)$-tight and \emph{bad} otherwise. Now, following our decision tree, we have
	\begin{eqnarray}
	\P(P\in I) &\le& \P(\text{reaching a bad leaf} ) + \sum_{\rho\text{ is a good leaf} }\P(\text{reaching } \rho \text{ and } P_{\rho} \in I)\nonumber\\
	&\le& (1-\frac{1}{2C^{d}})^{M} + \sum_{\rho\text{ is a good leaf} }\P(\text{reaching } \rho \text{ and } P_{\rho} \in I)\nonumber\\
	&\le&  2\exp\left (-\frac{r\tau}{4\alpha C^d}\right )+ \sum_{\rho\text{ is a good leaf} }\P(\text{reaching } \rho \text{ and } P_{\rho} \in I)\label{mon1}.
	\end{eqnarray}
	
Now, for each good leaf $\rho$, $P_\rho$ is either $(C,\beta)$-tight or $\tau'$-regular. Let $S$ be the set of indices $i$ of the internal nodes $\xi_i$ that lead to $\rho$. In other words, $\rho$ fixes $S$. Since the depth of the decision tree is at most $M\frac{\alpha}{\tau}\le \frac{r}{2}$, one has $|S|\le \frac{r}{2}$ and so  $q_\rho$ contains at least $r/2$ monomials of degree $d$ each, with mutually disjoint sets of random variables, and with coefficients at least 1 in magnitude. Therefore, $\Var_{(\xi_i) _{ i\notin S}} (P_\rho)=\Var_{(\xi_i)_{ i\notin S}}  (q_\rho)\ge r/2$.  
	
	Assume $P_\rho$ is $(C,\beta)$-tight, then by \eqref{q1}, one has $|P^*(\rho)|=\Omega(\sqrt r)\ge 2$. This together with \eqref{qp*1} give 
\begin{eqnarray}
\P(\text{reaching } \rho \text{ and } P_{\rho} \in I) &=& \P_{\xi_i, i\in S}(\text{reaching } \rho)\P_{\xi_i, i\notin S}(P_\rho \in I)\nonumber\\
&\le& \P_{\xi_i, i\in S}(\text{reaching } \rho)\P_{\xi_i, i\notin S}(|q_\rho| \ge |P^*(\rho)|-1>\frac{1}{2}|P^*(\rho)|)\nonumber\\
&\le& \beta \P_{\xi_i, i\in S}(\text{reaching } \rho) = \frac{1}{r} \P_{\xi_i, i\in S}(\text{reaching } \rho).\label{mon2}
\end{eqnarray}
Next, assume that $P_\rho$ is $\tau'$-regular. By Proposition \ref{regular},
\begin{eqnarray}
\P(\text{reaching } \rho \text{ and } P_{\rho} \in I) &=& \P_{\xi_i, i\in S}(\text{reaching } \rho)\P_{\xi_i, i\notin S}(P_\rho \in I)\nonumber\\
&\le& \P_{\xi_i, i\in S}(\text{reaching } \rho) \left (\frac{Cd}{r^{1/2d}} +Cd\tau'^{1/(4d+1)} \right ) \nonumber\\
&\le& \P_{\xi_i, i\in S}(\text{reaching } \rho) \left (\frac{Cd}{r^{1/2d}} +  C'd^{4/3}\tau^{1/(4d+1)}\left(\log \frac{1}{\tau}\right)^{1/4}\right )\label{mon3}
\end{eqnarray}

Since the events that the root $P$ reaches different leaves on the tree are disjoint, from \eqref{mon1}, \eqref{mon2}, and \eqref{mon3}, we get that for any $0<\tau<\frac{1}{3}$,
\begin{eqnarray}
\P(P \in I) &\le& 2\exp\left (-\frac{r\tau}{4C^{d+1}(d\log\log r+d\log d) }\right )+\frac{Cd}{r^{1/2d}}+ C'd^{4/3}\tau^{1/(4d+1)}\left(\log \frac{1}{\tau}\right)^{1/4}+ \frac{1}{r}.\label{tau1}
\end{eqnarray}

Set $\tau = \frac{8C^{d+1}\log r(d\log\log r+d\log d) }{r}$ then $\tau<\frac{1}{3}$ because we assumed that $d\le \frac{2\log r}{\log\log r}$. The first term on the right of \eqref{tau1} becomes $2r^{-2}$ and the third term is bounded from above by 
$B\frac{d^{4/3}\log ^{1/2}r}{r^{1/(4d+1)}}$. This completes the proof of the first bound.

\subsection{Second bound}
We next build on the arguments in the previous section to prove the second bound in Theorem \ref{theorem:LOpoly1}.

The main ingredient in proving the second bound is the following technical lemma of \cite{kane2014pseudorandom} which says that a random restriction of a sufficiently regular polynomial will likely have a much larger expectation compared to its standard-deviation. This is useful because polynomials with large expectation relative to standard-deviation have small probability of vanishing by tail bounds such as Theorem \ref{thm6}. In case the tail bound does not give a sufficiently good bound, we recurse on the new restricted polynomial. To state the lemma we need the following definition: 
For $\gamma \geq 0$, call a polynomial $P:\R^n \to \R$ $\gamma$-\textit{spread} if $\Var(P(\xi_1, \dots, \xi_n))^{1/2} \geq |\E(P(\xi_1, \dots, \xi_n))|/\gamma$.

\begin{proposition}\label{lm:goodblocks}
Let $b, n$ be such that $b | n$. Let $P:\R^n \to \R$ be a non-constant $\tau$-regular degree $d$ polynomial. Let $S_1,\ldots,S_b$ be a partition of $[n]$ into equal-sized blocks. For $\ell \in [b]$, and an assignment $\xi^l \in \{1, -1\}^{[n] \setminus S_\ell}$ to the variables not in $S_\ell$, let $P_{\xi^\ell}:\R^{S_\ell} \to \R$ denote the polynomial obtained by fixing the variables not in $S_\ell$ to $\xi^l$. Then,
$$\sum_{l=1}^b \P_{\xi^l}(P_{\xi^l} \text{ is $\gamma$-spread}) \leq 2^{O(d)} \cdot (\gamma^2 + 1) \cdot \left (\sqrt{b} + b \tau^{1/8d}\right), $$
where for clarity, the assignments $\xi^{l}$ for different $l$ are independent.

In particular, there exists an index $l \in [b]$, such that 
$$\P_{\xi^l}(P_{\xi^l} \text{ is $\gamma$-spread}) \leq 2^{O(d)} \cdot (\gamma^2+ 1) \cdot \left (1/\sqrt{b} +  \tau^{1/8d}\right ).$$
\end{proposition}

For the proof, we need the following definitions from \cite{kane2014correct}:
\begin{itemize}
\item For a function $f:\R^n \to \R$ and a vector $v \in \R^n$, $D_v f(x) = v \cdot \nabla f(x)$. 
\item Let $\zeta = (\zeta_1,\ldots,\zeta_n)$ and $\xi = (\xi_1,\ldots,\xi_n)$ be independent collections of Rademacher random variables. For a polynomial $P:\R^n \to \R$, define 
$$\alpha(P) = \E_{\zeta, \xi}\left(\min\left(1, \frac{|D_{\zeta} P(\xi)|^2}{|P(\xi)|^2}\right)\right).$$
\end{itemize}
The following claims are implicit in \cite{kane2014correct}.
\begin{lemma}\label{lm:alphaspread}
For any polynomial $P:\R^n \to \R$, $\Var(P) \leq 2^{O(d)} (\E(P)^2 + \Var(P)) \cdot \alpha(P)$. 
\end{lemma}
\begin{proof}
The claim is proved in \cite[Lemma 21]{kane2014correct}.
\end{proof}
\begin{lemma}\label{lm:alphabound}
Let $b, n$ be such that $b | n$. Let $P:\R^n \to \R$ be a non-constant $\tau$-regular degree $d$ polynomial. Let $S_1,\ldots,S_b$ be a partition of $[n]$ into equal-sized blocks. For $\ell \in [b]$, and an assignment $\xi^l \in \{1, -1\}^{[n] \setminus S_\ell}$ to the variables not in $S_\ell$, let $P_{\xi^\ell}:\R^{S_\ell} \to \R$ denote the polynomial obtained by fixing the variables not in $S_\ell$ to $\xi^l$. Then,
\begin{equation}
\sum_{\ell=1}^b \E_{\xi^\ell}(\alpha(P_{\xi^\ell})) = O(d^3 \alpha(P) \sqrt{b} + d^4 b \tau^{1/(8d)}),\label{kan}
\end{equation}
where for clarity, the assignments $\xi^{l}$ for different $l$ are independent.
\end{lemma}
\begin{proof}
Notice that the right-hand side of \eqref{kan} doesn't change if the assignments $\xi^{l}$ are obtained by choosing $n$ random variables $\xi_1, \dots, \xi_n$ and then looking at the $b$ different restrictions $\xi^{l}$.  The lemma is then proved in \cite[Proposition 19]{kane2014correct} (essentially Equation (4)). 
\end{proof}

Combining the above two claims gives us the proposition.
\begin{proof}[Proof of Proposition \ref{lm:goodblocks}]
For any index $\ell \in [b]$, we have
\begin{align*}
\P(P_{\xi^\ell} \text{ is $\gamma$-spread}) &= \P(\gamma^2 \Var(P_{\xi^\ell}) \geq \E(P_{\xi^\ell})^2)\\
&= \P\left(\frac{\Var(P_{\xi^\ell})}{\E(P_{\xi^\ell})^2 + \Var(P_{\xi^\ell})} \geq \frac{1}{\gamma^2 + 1}\right)\\
&\leq \P(\alpha(P_{\xi^\ell}) 2^{O(d)} \geq 1/(\gamma^2 + 1)) \text{ (by Lemma \ref{lm:alphaspread} applied to $P_{\xi^\ell}$)}\\
&\leq 2^{O(d)} \cdot (\gamma^2 + 1) \cdot \E(\alpha(P_{\xi^\ell})) \text{ (by Markov's inequality)}.
\end{align*}
Therefore, by Lemma \ref{lm:alphabound},
\begin{align*}
\sum_{\ell=1}^b \P(P_{\xi^\ell} \text{ is $\gamma$-spread}) &\leq 2^{O(d)} \cdot (\gamma^2 + 1) \cdot \sum_{\ell=1}^b\E(\alpha(P_{\xi^\ell})) \\
&= 2^{O(d)} \cdot (\gamma^2 + 1) \cdot O(d^3 \alpha(P) \sqrt{b} + d^4 b \tau^{1/(8d)})\\
&= 2^{O(d)} \cdot (\gamma^2 + 1) \cdot (\alpha(P) \sqrt{b} + b \tau^{1/8d}).
\end{align*}
The claim now follows as $\alpha(P) \leq 1$ by definition. 
\end{proof}

We are now ready to prove the second bound of Theorem \ref{theorem:LOpoly1}.
Similar to the proof of the first bound, without loss of generality, we can assume that $I = [-1, 1]$, $r$ is sufficiently large, and that $d\le \frac{\sqrt{\log r}}{\log \log r}$. 
Let, 
\begin{equation}
f(r,d) = \max\{\P(P(\xi)\in I): \text{ $P$ degree $d$ polynomial with $\rank(P) \geq r$}\}.
\end{equation}

Let $P$ be a degree $d$ multi-linear polynomial with $\rank(P) = r$ achieving the minimum $f(r,d)$. For fixed parameters $\tau\in (0, 1/3)$ and $\gamma> 2$ to be chosen later, let $\beta = \frac{1}{r}$ and let $\mathcal{T}$ be a decision tree as guaranteed by Proposition \ref{regularity lemma} with $M = \lceil\frac{r\tau}{2\alpha}\rceil$ where $\alpha$ and $\tau'$ are as in that Proposition. Then the depth of the tree is at most $\frac{r}{2}$, and as in the proof of the first bound, 
\begin{equation}\label{eq:mainproof1}
\P(P(\xi) \in I) \leq 2\exp\left (-\frac{r\tau}{4C^{d}\alpha}\right ) + \frac{1}{r} + \P[P_\rho(\xi) \in I | \text{ $P_\rho$ is $\tau'$-regular}].
\end{equation}

Now, consider a leaf $\rho$ so that $Q \equiv P_\rho$ is $\tau'$-regular.  Note that $\rank(Q) \geq r/2$ and in particular $Q$ is non-constant. Fix $b < r/4$, a parameter to be chosen later. Fix a partition $S_1,\ldots,S_b$ of the variables of $Q$ such that for $\ell \in [b]$, the restricted polynomials $Q^\ell$ obtained by fixing the variables not in $S_\ell$ each satisfy $\rank(Q^\ell) \geq \lfloor \rank(Q)/b\rfloor$ (this can be done for instance by first partitioning the variables witnessing $\rank(Q)$). Note that if the number of variables in $Q$ is not divisible by $b$, we only need to add a few variables to $Q$ without affecting its output nor its regularity. Now, by Proposition \ref{lm:goodblocks} applied to the polynomial $Q$, there exists $\ell \in [b]$ such that the polynomial $Q^\ell$ obtained by a random assignment to the variables not in $A^\ell$ is $\gamma$-spread with probability at most 
$$2^{O(d)} \cdot (\gamma^2 + 1) \cdot \left ({1/\sqrt{b} +  \tau'^{1/8d}}\right ).$$

Therefore,
\begin{eqnarray*}
\P(Q(y) \in I) &\leq& 2^{O(d)} \cdot (\gamma^2 + 1) \cdot \left ({1/\sqrt{b} + \tau'^{1/8d}}\right )  \cdot \P( Q^\ell(z) \in I |\text{ $Q^\ell$ is $\gamma$-spread})+ \\ 
&&\P(Q^\ell(z) \in I | \text{ $Q^\ell$ is not $\gamma$-spread}) \\
&\leq& 2^{O(d)} \cdot (\gamma^2 + 1) \cdot \left ({1/\sqrt{b} + \tau'^{1/8d}}\right )  \cdot f(\lfloor \rank(Q)/b\rfloor,d) + \P(Q^\ell(z) \in I | \text{ $Q^\ell$ is not $\gamma$-spread}).
\end{eqnarray*}

Finally, to bound the last term, observe that if $Q^\ell$ is not $\gamma$-spread and not identically zero, then
\begin{eqnarray*}
\P({Q^\ell(z) \in I}) &=& \P(|Q^{\ell}|\le 1) \leq \P({\left|Q^\ell(z) - \E(Q^\ell)\right| \geq |\E(Q^\ell)|-1})\\
&\leq& \P\left ({\left|Q^\ell(z) - \E(Q^\ell)\right| \geq \frac{\gamma \Var(Q^\ell)^{1/2}}{2}}\right )\\
&\leq& 2 \exp\left ({- \Omega(1) \gamma^{2/d}}\right ) \text{ (by Theorem \ref{thm6})},
\end{eqnarray*}
where in the next to last inequality, we use the inequalities $|\E(Q^{\ell})|\ge \gamma. \Var(Q^{\ell})^{1/2}\ge \gamma .\rank (Q^{\ell})^{1/2}\ge \gamma.(r/2b)^{1/2}\ge 2$ and so $|\E (Q^{\ell})|-1\ge \frac{|\E (Q^{\ell})|}{2}\ge \frac{\gamma \Var(Q^{\ell})^{1/2}}{2}$.

Combining the above arguments, we get that if  $b \leq r/4$, 
$$\P({Q(x) \in I}) \leq 2^{O(d)} \cdot (\gamma^2 + 1) \cdot \left ({1/\sqrt{b} + \tau'^{1/8d}}\right ) \cdot f(\lfloor r/b\rfloor,d) + O(1) \exp\left ({- \Omega(1) \gamma^{1/2d}}\right ) .$$
Hence, by \eqref{eq:mainproof1} we have that 
\begin{equation}
\P({P(x) \in I}) \leq 2\exp\left (-\frac{r\tau}{4C^{d}\alpha}\right ) + \frac{1}{r} + 2^{O(d)} \cdot (\gamma^2 + 1) \cdot \left ({1/\sqrt{b} +  \tau'^{1/8d}}\right ) \cdot f(\lfloor r/b\rfloor,d) + O(1) \exp\left ({- \Omega(1) \gamma^{2/d}}\right ).\label{rec}
\end{equation}

Now, as in the proof of the first bound of Theorem \ref{theorem:LOpoly1}, set $\tau = \frac{8C^{d+1}\log r(d\log\log r+d\log d) }{r}$, $b = r^{1/4d}/(d\log r)^{Cd}$, and $\gamma = (C \log r)^{d/2}$. Then, 
$$f(r,d) \leq (C \log r))^{Cd} \cdot f(r^{1-1/4d},d) \cdot r^{-1/8d}.$$
(here we used the fact that $f(r, d)\ge \Omega (r^{-1/2})$ by choosing the polynomial $p(\xi_1, \dots, \xi_{rd})=\xi_1\xi_2\dots\xi_d+\xi_{d+1}\dots\xi_{2d}+\dots + \xi_{rd-d+1}\dots\xi_{rd}$, and so all the other terms on the right-high side of \eqref{rec} are dominated by the term $(C \log r))^{Cd} \cdot f(r^{1-1/4d},d) \cdot r^{-1/8d}$.)

Let $a = 1 - 1/4d$. Applying this recurrence relation $k$ times with $r^{a^{k}} = C$ (so $k  = \Theta(d\log\log r)$), we get
\begin{eqnarray}
f(r,d) &\leq& (C \log r))^{kCd} \left (\prod_{i=0}^{k-1}a^{i}\right )^{Cd}\cdot f(r^{a^{k}},d) \cdot r^{-(\sum_{i=0}^{k-1}a^{i})/8d}\nonumber\\
&\le& e^{O(d^{2}(\log \log r)^{2})}r^{-(1-a^{k})/2} = Ce^{O(d^{2}(\log \log r)^{2})}r^{-1/2},\nonumber
\end{eqnarray}
completing the proof of the second bound and hence Theorem \ref{theorem:LOpoly1}.


\section{General distributions}\label{gen-proof}
\subsection{Proof of Theorem \ref{th:littlewoodoffordbiasedmain}}

We reduce the $p$-biased case to the uniform distribution at the expense of a loss in the rank of the polynomial and then apply Theorem \ref{theorem:LOpoly1}. 

First notice that if $x\sim \mu_p$, then $1-x\sim \mu_{1-p}$. And so, by replacing the polynomial $P$ by $Q(x_1, \dots, x_n) = P(1- x_1, \dots, 1-x_n)$, we can exchange the roles of $p$ and $1-p$. Therefore, without loss of generality, we assume that $\alpha = p\le 1/2$. 

Our assumption $2^{d}p^{d}r\ge 3$ guarantees that $\log \log (2^{d}p^{d}r) = \Omega(1)$ and hence by choosing the implicit constants on the right-hand side of Theorem \ref{th:littlewoodoffordbiasedmain} to be sufficiently large, we can assume that $2^{d}p^{d}r$ is greater than 100 (say).

Let $\eta_1, \dots, \eta_n$ and $\xi_1', \dots, \xi_n'$ be independent Bernoulli random variables with $\P(\eta_i=0) = 1/2$ and $\P(\xi_i' = 0) = 1-2p$. Let $\xi_i = \eta_i\xi_i'$ then $\xi_1, \dots, \xi_n$ are iid Bernoulli variables with $\P(\xi_i = 0) = 1-p$. Therefore, we need to bound $\P(P(\xi_1, \dots, \xi_n)\in I)$.

From the definition of $\rank(P)$, there exist disjoint sets $S_1, \dots, S_r$ such that $|a_{S_j}|\ge 1$ for all $j=1, \dots, r$. We have $P(\xi_1, \dots, \xi_n) = \sum_{S\subset [n], |S| \le d} \left (a_{S}\prod_{i\in S}\xi_i'\right )\prod_{i\in S}\eta_i$. Conditioning on the $\xi_i'$'s, $P$  becomes a polynomial of degree $d$ in terms of $\eta_i$ whose coefficients associated with $S_j$ are $b_{S_j}:=a_{S_j}\prod_{i\in S_j}\xi_i'$ accordingly. For each such $j$, one has 
\begin{eqnarray}
\P_{\xi_1', \dots, \xi_n'}(|b_{S_j}|\ge 1)= \P(\xi_i'=1, \forall i\in S_j) = (2p)^{d}.\nonumber
\end{eqnarray}

Now, since the sets $S_j$ are disjoint, the events $|b_{S_j}|\ge 1$ are independent. Define 
 $X = \sum_{j=1, \dots, r}\textbf{1}_{|b_{S_j}|\ge 1}$. By the classical Chernoff's bound we have,  for $0<\gamma<1$, $\P(|X-\E X| \ge \gamma\E X) \le 2e^{-\gamma^{2}\E X/3}$. 
  Thus, we  conclude that with probability at least $1 - \exp(-2^{d-1}p^{d}r/6)$, there are at least $2^{d-1}p^{d}r$ indices $j$ with $|b_j|\ge 1$. Conditioning on this event, we obtain a polynomial of degree $d$ in terms of $\eta_1, \dots, \eta_n$ which has rank at least $2^{d-1}p^{d}r$. The theorem now follows from applying Theorem \ref{theorem:LOpoly1} to this polynomial and noting that the additional error of $\exp(-2^{d-1}p^{d}r/6)$ is smaller than both terms from Theorem \ref{theorem:LOpoly1}. 

\subsection{Proof of Theorem \ref{thm:generalist}}
By replacing $P(x_1, \dots, x_n)$ by $Q(x_1, \dots, x_n) = P(x_1+y_1, \dots, x_n+y_n)$ and $\xi_i$ by $\xi_i-y_i$, we can also assume without loss of generality that $y_i=0$ for all $i$. Furthermore, we can assume that $\P(\xi_i\le 0)=p$ for all $i$. Indeed, if for some $i$, $\P(\xi_i>0)=p$, we replace $\xi_i$ by $-\xi_i$ and modify  the polynomial $P$ accordingly to reduce to the case $\P(\xi_i<0)=p$. And then the proof runs along the same lines as in the case $\P(\xi_i\le 0)=0$. 

For each $i=1, \dots, n$, let $\xi_i^{+}$ and $\xi_i^{-}$ be independent random variables satisfying $\P(\xi_i^{+}\in A) = \P(\xi_i\in A|\xi_i>0)$ and $\P(\xi_i^{-}\in A) = \P(\xi_i\in A|\xi_i\le 0)$ for all measurable subset $A\subset \R$. Let $\eta_1, \dots, \eta_n$ be iid random Bernoulli variables (independent of all previous random variables) such that $\P(\eta_i = 0) = p$.
Let $\xi_i' = \eta_{i}\xi_i^{+} + (1-\eta_{i})\xi_i^{-}$, then $\xi_i'$ and $\xi$ have the same distribution. Therefore, it suffices to bound the probability that $P(\xi_1', \dots, \xi_n')$ belongs to $I$. One has
\begin{equation}
P(\xi_1', \dots, \xi_n') = P(\eta_1(\xi_1^{+}-\xi_1^{-}) + \xi_1^{-}, \dots, \eta_n(\xi_n^{+}-\xi_n^{-}) + \xi_n^{-}) = \sum_{S\subset [n], |S| = d} \left (a_{S}\prod_{i\in S}(\xi_i^{+}-\xi_i^{-})\right )\prod_{i\in S}\eta_i + Q,\nonumber
\end{equation}
where $Q$ is some polynomial which has degree $<d$ in terms of $\eta_i$ when all the $\xi_i^{\pm}$ are fixed. From the definition of $\rank(P)$, let $S_1, \dots, S_r$ be disjoint subsets of $[n]$ with $|a_{S_j}|\ge 1$ for all $1\le j\le r$. Conditioning on the variables $\xi_i^{\pm}$, the polynomial $P$ becomes a polynomial of degree $d$ in terms of $\eta_i$ whose coefficients associated with $S_j$ are $b_{S_j}:=a_{S_j}\prod_{i\in S_j}(\xi_i^{+}-\xi_i^{-})$ accordingly. For each such $j$, one has 
\begin{eqnarray}
\P_{\xi_1^{\pm}, \dots, \xi_n^{\pm}}(|b_{S_j}|\ge 1)\ge \P(\xi_i^{+}-\xi_i^{-}\ge 1, \forall i\in S_j).\nonumber
\end{eqnarray}

Since $\xi_{i}^{+}\ge 0\ge \xi_i^{-}$ a.e., one has $2\P(\xi_i^{+}-\xi_i^{-}\ge 1)\ge \P(|\xi_i^{+}\ge 1) + \P(|\xi_i^{-}\le -1) = \P(|\xi_i|\ge 1)\ge \epsilon$. Hence, 
\begin{equation}
\P_{\xi_1^{\pm}, \dots, \xi_n^{\pm}}(|b_{S_j}|\ge 1)\ge 2^{-d}\epsilon^{d}.\nonumber
\end{equation}

Now, since the sets $S_j$ are disjoint, the events $|b_{S_j}|\ge 1$ are independent. Therefore, using a Chernoff-type bound as in the proof of Theorem \ref{th:littlewoodoffordbiasedmain}, one can conclude that with probability at least $1 - \exp(-2^{-d}\epsilon^{d}r/12)$, there are at least $r2^{-d}\epsilon^{d}/2$ indices $j$ with $|b_j|\ge 1$. Conditioning on this event, we obtain a polynomial of degree $d$ in terms of $\eta_1, \dots, \eta_n$ which has rank at least $r2^{-d}\epsilon^{d}/2$. Using Theorem \ref{th:littlewoodoffordbiasedmain}, one obtains the desired bound.

\section{Proof of Theorem \ref{th:orapprox}} \label{app-proof}
\newcommand{\zero}{\overline{0}}

Let $a$ be an integer to be chosen later. Let $D = \lfloor \log_a (\log_2 n-1) \rfloor$ be the largest integer such that $2^{-a^D} \geq 2/n$. Let $\mu$ be the distribution obtained by the following procedure:
\begin{enumerate}
\item With probability $1/2$ output $x =\zero$ (the all $0$'s vector).
\item With probability $1/2$ pick an index $i \in \{1,\ldots,D\}$ uniformly at random and output $x \sim \mu_{2^{-a^i}}^n$. 
\end{enumerate}
We next show that for some constant $c > 0$, there exists no polynomial $P$ of degree $d < c (\log \log n)/(\log \log \log n))$ such that $\P_{x \sim \mu}(P(x) = OR(x)) \geq 2/3$. Let $P$ be such a polynomial. Then, necessarily, $P(\zero) = 0$; as $\P_{x\sim \mu}(P(x)=0) \le 1/2 + 1/2(1-2^{-a^{D}})^n \leq 1/2 + (1/2) (1-2/n)^n <2/3$, there must exist a set of indices $I \subseteq [D]$ with $|I| \geq \Omega(D)$ such that for all $i \in I$, 
$$\P_{x \sim \mu_{2^{-a^i}}}(P(x) = 1) = \Omega(1).$$

Let $I = \{i_1 < i_2 < \cdots < i_k \}$ and for $\ell \in [k]$, let $p_\ell = 2^{-a^{i_\ell}}$. Now, by Theorem \ref{th:littlewoodoffordbiasedmain} applied to the polynomial $P-1$ and $x \sim \mu_{p_1}^n$, we get that either $\rank(P) \leq (3/2p_1)^d$ or 
$$\Omega(1) = \P(P(x) = 1) \leq O(d^{4/3}) \frac{\log(\rank(P) (2p_1)^d)^{1/2}}{(\rank(P) (2p_1)^d)^{1/(4d+1)}}.$$
Hence, in any case, $\rank(P) \leq r_1 = (d)^{O(d)}/p_1^d$. This in turn implies that there exists a set of $r_1 \cdot d$ indices $S_1\subseteq [n]$ such that the polynomial $P_1 = P_{S_1}$ obtained by assigning the variables in $S_1$ to $0$ is of degree at most $d-1$. Further, for $x \sim \mu_{p_{2}}^{[n]}$, 
\begin{eqnarray*}
\Omega(1) &=& \P_{x}(P(x) = 1) = \P(x_{S_1} = 0) \cdot \P_{x}(P(x) = 1 | x_{S_1} = 0) + \P(x_{S_1} \neq 0) \cdot \P_x(P(x) = 1 | x_{S_1} \neq 0)\\
&\leq& \P_{x \sim \mu_{p_{2}}^{[n] \setminus [S_1]}}(P_1(x) = 1) + \P(x_{S_1} \neq 0)\\
&\leq& \P_{x \sim \mu_{p_{2}}^{[n] \setminus [S_1]}}(P_1(x) = 1) + |S_1| \cdot p_{2}.
\end{eqnarray*}
Thus, 
$$\P_{x \sim \mu_{p_{2}}^{[n] \setminus [S_1]}}(P_1(x) = 1) \geq \Omega(1) - d^{O(d) + 1} \cdot (p_2/p_1^d) = \Omega(1) - d^{O(d)+1} 2^{-a^{i_2} + d a^{i_1}} \geq \Omega(1) - d^{O(d)} 2^{-a^{i_1}},$$
for $a \geq 2d$. Further, note that $P_1(\zero) = 0$.

Iterating the argument with $P_1$ and so forth, we get a sequence of polynomials $P_1, P_{2}, \ldots,P_{k-1}$ such that for $1 \leq j  \leq \min(d,k-1)$, $P_{j}$ is of degree at most $d-j$, $P_{j}(\zero) = 0$ and for $x \sim \mu_{p_{j+1}}^{[n] \setminus (S_1 \cup \cdots \cup S_{j})}$, 
$$\P_x(P_{j}(x) = 1) = \Omega(1) - d^{O(d)+j} 2^{-a}.$$
This clearly leads to a contradiction if $k > d$ and $a \geq C d \log d$ for a large enough constant $C$ (so that the right hand side of the above equation is non-zero for $j = d$).

Therefore, setting $a = C d \log d$, for a sufficiently big constant $C$, we must have $k = \Omega(D) \leq d$. That is, $\log_2 (n-1) = a^{O(d)} = d^{O(d)}$. Thus, we must have $d = \Omega(1) (\log \log n)/(\log \log \log n)$.
\bibliographystyle{siam}
\bibliography{ref}

\end{document}